\documentclass[11pt]{article}
\usepackage{amssymb,amsmath,amsthm,graphicx}
\usepackage{color,enumerate,hyperref}
\usepackage{cite,fullpage}
\usepackage{cleveref}
\usepackage{caption}
\usepackage{subcaption}
\usepackage{relsize}
\usepackage{rotating}
\usepackage[colorinlistoftodos]{todonotes}

\newtheorem{theorem}{Theorem}[section]
\newtheorem{lemma}[theorem]{Lemma}
\newtheorem{corollary}[theorem]{Corollary}
\newtheorem{proposition}[theorem]{Proposition}

\newtheorem{assumption}{Assumption}
\newtheorem{claim}{Claim}

\newtheorem{example}[theorem]{Example}

\theoremstyle{definition}

\newcommand\diam{\operatorname{diam}}
\newcommand\dist{\operatorname{dist}}

\newcommand{\edalong}[1]{{\sc Exact Diameter-$#1$ Augmentation}}
\newcommand{\dalong}[1]{{\sc Diameter-$#1$ Augmentation}}
\newcommand{\medalong}[1]{{\sc Minimum Exact Diameter-$#1$ Augmentation}}
\newcommand{\mdalong}[1]{{\sc Minimum Diameter-$#1$ Augmentation}}
\newcommand{\mdilong}[1]{{\sc Minimum Distance-$#1$ Augmentation}}
\newcommand{\eda}[1]{{\sc $#1$-eda}}
\newcommand{\da}[1]{{\sc $#1$-da}}
\newcommand{\meda}[1]{{\sc $#1$-meda}}
\newcommand{\mda}[1]{{\sc $#1$-mda}}
\newcommand{\mdi}[1]{{\sc $#1$-mdisa}}

\usepackage{amssymb}
\usepackage{graphicx}
\begin{document}

\title{Complexity and algorithms for constant diameter augmentation problems}

\author{Eun Jung Kim\footnotemark[1] ,
 Martin Milani{\v c}\footnotemark[2] ,
J\'er\^ome Monnot\footnotemark[1] ,
Christophe Picouleau\footnotemark[3]}

\def\thefootnote{\fnsymbol{footnote}}

\footnotetext[1]{ \noindent
Universit\'e Paris-Dauphine, Paris, France.
Email: \texttt{eunjungkim78@gmail.com}}

\footnotetext[2]{ \noindent
FAMNIT and IAM, University of Primorska, Koper, Slovenia. Email: \texttt{martin.milanic@upr.si}}

\footnotetext[3]{ \noindent
Conservatoire National des Arts et M\'etiers, CEDRIC laboratory, Paris, France. Email: {\tt
christophe.picouleau@cnam.fr}}

\date{\today}

\maketitle

\begin{abstract}
We study the following problem: for given integers $d,k$ and graph~$G$, can we obtain a graph with diameter $d$ via at most $k$ edge deletions\,? We determine the computational complexity of this and related problems for different values of $d$.
\end{abstract}


\bigskip
\begin{center}
{\it To the memory of our colleague and friend J\'er\^ome.}
\end{center}

\section{Introduction}
\label{s-intro}

A typical graph modification problem aims to modify a graph parameter $\pi$, via a small number of operations. The operations involved may be vertex deletion, edge deletion, edge addition, edge contraction, etc.  That is, for a fixed graph operation,
we ask, given a graph $G$ and two integers~$k$ and~$d$, whether $G$ can be transformed into a graph $G'$ by using at most $k$ operations so that $\pi(G')\leq \pi(G)-d$ or $\pi(G')\geq \pi(G)+d$. Which inequality we want depends on the behavior of parameter $\pi$ under the fixed operation: we require $\pi(G')\leq \pi(G)-d$ whenever $\pi(G')\leq \pi(G)$ holds for any graph $G'$ obtained from $G$ by applying the given operation any number of times, and, similarly, $\pi(G')\geq \pi(G)+d$ whenever the inequality $\pi(G')\geq \pi(G)$ is guaranteed. Typically, $\pi(G)$ is defined as the optimal value of an instance of an optimization problem on graphs. Thus, such problems are called {\it blocker problems}, as the set of vertices or edges involved in the operations ``block''  the optimal value of the problem. Identifying the part of the graph responsible for a significant change of the parameter under consideration gives crucial information on the graph.

Blocker problems have been given much attention over the
last few years \cite{Bentz,CWP11,T10,BBKP18}.
Graph parameters considered were the matching number \cite{RBPDCZ10,ZRPWCB09},
the $s,t$-shortest path, the maximum flow and minimum $s,t$-cut values~\cite{CWP11}, the chromatic number \cite{BBPR,DPPR15}, the independence number \cite{BBPR,BTT11,DPPR15}, the clique number \cite{PBP14,DPPR15}, the domination number \cite{GPR19,PBP15}, and the vertex cover number \cite{BTT11,LSSS20,PBP15}, for instance. Concerning the operations involved to the minimal modification of a graph parameter or a graph property the reader can see \cite{DPPR15,LY80} for node deletion, \cite{PBP15,WTN83} for edge deletion, \cite{DPPR15,GPR19,HHLP,LSSS20,WTN83} for edge contraction.

Here, we aim to determine the number of edges to delete from a given graph to increase the diameter of the graph by a fixed amount. Our goal is to determine the computational complexity, that is, to either find polynomial-time algorithms or prove {\sf NP}-completeness, for several related problems.

We consider the following decision problems:

\begin{center}
\fbox{\parbox{0.95\linewidth}{\noindent
\edalong{d} (\eda{d}):\\[.8ex]
\begin{tabular*}{.93\textwidth}{rl}
{\em Input:} & A connected graph $G$.\\
{\em Question:} & Is there a set $F\subseteq E(G)$ such that $G-F$ is a connected graph of diameter $d$\,?
\end{tabular*}
}}
\end{center}

\begin{center}
\fbox{\parbox{0.95\linewidth}{\noindent
\dalong{d} (\da{d}):\\[.8ex]
\begin{tabular*}{.93\textwidth}{rl}
{\em Input:} & A connected graph $G$.\\
{\em Question:} & Is there a set $F\subseteq E(G)$ such that $G-F$ is a connected graph of diameter \\
& at least $d$\,?
\end{tabular*}
}}
\end{center}

\begin{center}
\fbox{\parbox{0.95\linewidth}{\noindent
\medalong{d} (\meda{d})\\[.8ex]
\begin{tabular*}{.93\textwidth}{rl}
{\em Input:} & A connected graph $G$, a positive integer $k$.\\
{\em Question:} & Is there a set $F\subseteq E(G)$ such that $|F|\le k$ and $G-F$ is a connected graph\\
& of diameter $d$\,?
\end{tabular*}
}}
\end{center}

\begin{center}
\fbox{\parbox{0.95\linewidth}{\noindent
\mdalong{d} (\mda{d})\\[.8ex]
\begin{tabular*}{.93\textwidth}{rl}
{\em Input:} & A connected graph $G$, a positive integer $k$.\\
{\em Question:} & Is there a set $F\subseteq E(G)$ such that $|F|\le k$ and $G-F$ is a connected graph\\
& of diameter at least $d$\,?
\end{tabular*}
}}
\end{center}

In all the above problems we consider $d$ to be a fixed constant independent of the input graph. If $d$ is part of input, then all four problems are {\sf NP}-complete as they generalize the {\sc Hamiltonian Path} problem (which is obtained by taking $d = |V(G)|-1$ and, when necessary, $k = |E(G)|$), see~\cite{SBVL87}.

\medskip

{\noindent {\bf Our results and organization of the paper.}} In Section~\ref{prelim} we give the main definitions and notations we use throughout the paper together with some preliminary observations about the decisions problems defined above. In Sections \ref{smalldiam} and \ref{3meda} we give polynomial-time algorithms for the \mda{3}, \meda{3}, and \eda{3} problems. Section \ref{5meda} is devoted to {\sf NP}-completeness results for \meda{k}, $k\ge 5$. We show that \meda{5} is {\sf NP}-complete for graphs of diameter $d\in\{3,4\}$, and show the {\sf NP}-completeness of \meda{k} for all $k\ge 6,$ for some different values of the diameter of the input graph. In Section \ref{concl} we give a conclusion and propose some future research problems.

\medskip

{\noindent {\bf Related work.} There is a large literature studying extremal questions related to diameter (see, e.g.,~\cite{Chung87}).
Among the questions most relevant to our study, let us mention the effect on the diameter of deleting $t$ edges from a $(t+1)$-edge-connected graph~\cite{ChungGarey}, the maximum number of edges in an $n$-vertex diameter-$d$-critical graph (that is, a graph with diameter $d$ such that the diameter increases upon deletion of any edge; see~\cite{LM16}) and the minimum number of edges in an $n$-vertex graph such that deleting any edge results in a graph with diameter $d$ (see~\cite{Chung}). A connection between the \meda{3} problem and Moore graphs will be established in Section~\ref{3meda}. From the algorithmic point of view, related problems extensively studied in the literature include the
complementary problem of adding a minimum number of edges to achieve diameter $d$ (see, e.g.,~\cite{FGGM15,GHN13,IYN06,LMT92,SBVL87}) and the
length-bounded max flow and min cut problems (see~\cite{BEHKKPSS10,MM03}).

\section{Notations and preliminaries}\label{prelim}

We only consider finite, simple, and undirected graphs. We refer to \cite{West} for undefined terminology.

Let $G=(V,E)$ be a graph. For a subset $F\subseteq E$, we let $G-F$ denote the subgraph of $G$ with vertex set~$V$ and edge set $E\setminus F$. The subgraph of $G$ \emph{induced by} a set $S\subseteq V$ is denoted by $G[S]$ and defined as $G[S]=(S,F)$ where $F=\{xy\in E(G)\vert x,y\in X\}$. The length of a path is its number of edges. Given two vertices $u$ and $v$, $\dist_G(u,v)$ denotes the minimum length of a path between $u$ and $v$ (note that since $G$ is assumed to be connected $\dist_G(u,v)$ is finite). A $u,v$-path of minimum length is called a shortest $u,v$-path.  The \emph{diameter} of $G$ is the maximum length of a shortest path, that is, $\diam(G)=\max_{(u,v)\in V^2}\{\dist_G(u,v)\}$. A \emph{diametral path} in a graph $G$ is any shortest path between vertices $u$ and $v$ such that $\dist_G(u,v) = \diam(G)$.
A \emph{cut-vertex} in a connected graph $G$ is a vertex $v\in V(G)$
such that $G-v$ is not connected.
A graph is \emph{biconnected} if it is connected and has no cut-vertices.
A \emph{cut-edge} in a connected graph $G$ is an edge $e\in E(G)$ such that $G-e$ is not connected.
A vertex in a graph is \emph{universal} if it is adjacent to all other vertices. The \emph{girth} of a graph $G$, denoted $g(G)$, is the shortest length of a cycle in $G$ (or $\infty$ if $G$ is acyclic). A \emph{triangle} in a graph $G$ is a cycle of length three. Given two graphs $G$ and $H$, we say that $G$ is \emph{$H$-free} if no induced subgraph of $G$ is isomorphic to $H$. The \emph{$\mathcal{F}$-free graphs} where $\mathcal{F}$ is a finite set of graphs are defined analogously.

\subsection*{(Non)monotonicity}
It follows directly from the definitions that the \da{d} and \mda{d} problems satisfy a certain monotonicity,
in the sense that for every $d<d'$, every graph $G$ that is a yes instance to the \da{d'} problem
is also a yes instance to the \da{d} problem, and similarly, every pair $(G,k)$
that is a yes instance to the \mda{d'} problem is also a yes instance to the \mda{d} problem.
Unsurprisingly, the analogous monotonicity fails to hold for the \eda{d} and \meda{d} problems.
It is not difficult to construct graphs $G$ and pairs of positive integers $d$, $d'$, and $k$, with $d<d'$ such that
$G$ is a yes instance to \eda{d} as well as to \eda{d'}, while $(G,k)$ is a yes instance to \meda{d'} but not to \meda{d}.
In other words, there is a set of at most $k$ edges such that deleting them from $G$ results in a graph with diameter $d'$ but
there is no set of at most $k$ edges whose deletion would produce a graph with diameter $d$.

\begin{example}
The graph shown in Fig.~\ref{fig:example} has diameter $3$,
deleting any of the edges $v_1v_6$, $v_3v_4$, $v_4v_5$, $v_5v_6$ results in a graph with diameter $5$,
while deleting any other edge does not change the diameter.
However, a graph with diameter $4$ can be obtained by deleting, for example, the two edges $v_0v_1$ and $v_0v_3$.

\begin{figure}[h!]
  \centering
   \includegraphics[width=0.25\textwidth]{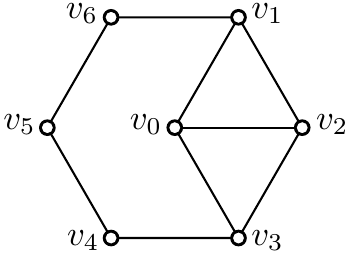}
\caption{A graph in which the diameter can be increased by two by deleting a single edge,
while two edges need to be deleted to increase the diameter by one.}\label{fig:example}
\end{figure}

\end{example}

\subsection*{The \da{d} problem}

We show how the \da{d} problem is related to the problem of searching for a path of length $d$. Then we deduce that it is polynomial-time solvable.

\begin{proposition}\label{prop:ed}
For a connected graph $G$ and a positive integer $d$, the following are equivalent.
\begin{enumerate}
  \item $G$ is a yes instance to the \da{d} problem.
  \item $G$ has a connected spanning subgraph of diameter at least $d$.
  \item $G$ has a spanning tree of diameter at least $d$.
  \item $G$ contains a path of length at least $d$.
\end{enumerate}
\end{proposition}

\begin{proof}
The equivalence between statements 1 and 2 is immediate from the definition of the \da{d} problem.
Clearly, statement 3 implies statement 2. Moreover, statement 2 implies statement 4, as if $G'$ is a
connected spanning subgraph of $G$ of diameter at least $d$, then any diametral path in $G'$ is a path in $G$ of length at least $d$.
Finally, we argue that statement 4 implies statement 3.
Let $P$ be a path in $G$ of length at least $d$. Then, $G$ has a depth-first traversal starting with the edges of $P$ along the path.
The corresponding DFS tree is a spanning tree of $G$ of diameter at least $d$, implying statement 3.
\end{proof}

Since for each fixed $d$, the presence of a path of length at least $d$ can be tested in polynomial time,
Proposition~\ref{prop:ed} implies the following.

\begin{corollary}\label{cor:da}
For every positive integer $d$, the \da{d} problem is solvable in polynomial time.
\end{corollary}

%Furthermore, there is an algorithm running in time $f(d)(|V(G)|^{\Oh(1)})$.\footnote{\color{red}Remark (Martin): EunJung, could you please add some details about this, or a reference? The corresponding FPT algorithm implies that the problem is polynomial also for some non-constant values of $d$. Which ones? (Not sure if it's worth mentioning.)}

One may wonder whether there are similar characterizations of yes instances for the \eda{d} problem as those given in Proposition~\ref{prop:ed}, in particular, whether the existence of a connected spanning subgraph of diameter $d$ implies the existence of a spanning tree of diameter $d$. While this implication clearly fails if $d$ is allowed to be the diameter of $G$ (take for instance any cycle of length at least~$4$), examples can also be constructed showing that the implication fails if $d<\diam(G)$.

\begin{example}
Let $G$ be the graph obtained from the $5$-cycle by adding to it a new vertex $v$ and connecting it by an edge to exactly two non-adjacent vertices of the $5$-cycle. Graph $G$ is of diameter two, it has a connected spanning subgraph of diameter three (for example, the subgraph obtained by deleting one of the two edges incident with $v$) but it is not difficult to verify that $G$ has no spanning tree of diameter three.
\end{example}

\subsection*{Complete graphs}

In the case when the input graph $G$ is complete, the \da{d}, \eda{d}, \mda{d}, and \meda{d} problems are either trivial or related to a known and solved question in extremal graph theory.

First, since the maximum diameter of an $n$-vertex connected graph is $n-1$, the complete graph $K_n$ is a yes instance to the \da{d} (or to the \eda{d}) problem if and only if \hbox{$n\ge d+1$}.

Second, the \meda{d} and the \mda{d} problems on $K_n$ are equivalent to the problems of deleting the smallest number of edges from $K_n$ so that the resulting graph is of diameter exactly $d$ (resp., at least $d$). Both of these questions are equivalent to a well-known problem in extremal graph theory asking for the maximum number of edges in an $n$-vertex graph of diameter $d$. O.~Ore proved that the maximum number of edges in a graph with $n$ vertices and diameter $d$ is upper bounded by $d+{1\over 2}(n-d-1)(n-d+4)$ and this bound is tight~\cite{Ore68}.
Thus, the answer to either \meda{d} or \mda{d} problem given $K_n$ and a positive integer $k$ is yes if and only if
$k\ge {n\choose 2} - d - {1\over 2}(n-d-1)(n-d+4)$.

For later use, we make note of the above observations.

\begin{theorem}[Ore]\label{thm:complete}
For all $d\ge 1$, the \eda{d}, \mda{d}, and \meda{d} problems are solvable in linear time for complete graphs.
\end{theorem}

\subsection*{The \eda{d}, \mda{d}, and \meda{d} problems for $d\in \{1,2\}$}

We give some observations solving the cases  $d \in \{1,2\}$.
First, note that for each of the \eda{d}, \mda{d}, and \meda{d} problems, we may assume that $\diam(G)<d$ since otherwise
the answer is trivial:
\begin{itemize}
  \item If $\diam(G)>d$, then $G$ is a no instance to the \eda{d} problem (since no spanning subgraph of $G$ has diameter exactly $d$).
  Similarly, any $(G,k)$ with $k\ge 0$ is a yes instance to the \mda{d} problem and a no instance to the \meda{d} problem.
  \item If $\diam(G)=d$, then $G$ is a yes instance to the \eda{d} problems (take $F = \emptyset$)
  and any $(G,k)$ with $k\ge 0$ is a yes instance to the \mda{d} and \meda{d} problems.
\end{itemize}

Since complete graphs are the only graphs with diameter less than $2$, the above observation implies that all three problems are polynomially solvable for $d \in \{1,2\}$.

\section{A polynomial-time algorithm for the \mda{3} problem}\label{smalldiam}

We prove the following.

\begin{theorem}\label{mda3}
The \mda{3} problem is solvable in polynomial time.
\end{theorem}

The algorithm is based on a polynomial-time algorithm for the following problem for the case $d = 3$.

\begin{center}
\fbox{\parbox{0.95\linewidth}{\noindent
\mdilong{d} (\mdi{d})\\[.8ex]
\begin{tabular*}{.93\textwidth}{rl}
{\em Input:} & A connected graph $G$, distinct vertices $x,y$ of $G$, a positive integer $k$.\\
{\em Question:} & Is there a set $F\subseteq E(G)$ such that $|F|\le k$,  $G-F$ is a connected graph\\
& and  $\dist_{G-F}(x,y)\geq d$\,?
\end{tabular*}
}}
\end{center}

 We note that the inapproximability of two optimization problems related to \mdi{d} when $d$ is not constant was established in~\cite{KBBEGRZ08}. Beier et al.~showed that for all $d\ge 4$, the optimization version of the \mdi{d} problem is {\sf NP}-hard to approximate within a factor of 1.1377~\cite{BEHKKPSS10}. On the other hand, Mahjoub and McCormick showed that even a weighted version of the \mdi{3} problem is solvable in polynomial time, by reducing it to a max flow problem~\cite{MM03}.
Hence, taking a weight $w_e = 1$ for each edge $e\in E(G)$ in their construction, we obtain the following.

\begin{theorem}[Mahjoub and McCormick]\label{thm:mdi3}
The \mdi{3} problem is solvable in polynomial time.
\end{theorem}

\begin{proof}[Proof of Theorem~\ref{mda3}.]
Let $(G,k)$ be an instance of the \mda{3} problem. Note that $(G,k)$ is a yes instance if and only if $G$ has a pair of distinct vertices $x,y$ and a set $F\subseteq E(G)$ such that $|F|\le k$, $G-F$ is connected and $\dist_{G-F}(x,y) \ge 3$. This last condition is equivalent to the existence of a pair $x,y\in V(G)$ of distinct vertices such that $(G,x,y,k)$ is a yes instance to \mdi{3}, which can be solved in polynomial time by Theorem~\ref{thm:mdi3}.
\end{proof}

\section{Polynomial-time algorithms for the \eda{3} and \meda{3} problems}\label{3meda}

In this section we show how to solve the \eda{3} and \meda{3} problems in polynomial time.

The following weight function on the edges of a graph $G$ will be useful. For each edge $e\in E(G)$, we denote by $w_G(e)$ the minimum length of a cycle in $G$ containing $e$ (or $\infty$ if no such cycle exists). The weight function $w_G$ is well-defined and can be computed in polynomial time; in fact, for every edge $e= uv\in E(G)$, the value of $w_G(e)$ equals the length of a shortest $u,v$-path in $G-e$ (or  $\infty$ if no such path exists).

\begin{lemma}\label{lem:weight}
Let $G$ be a graph of diameter two. Then, for every edge $e$ of $G$, we have $w_G(e)\in \{3,4,5,\infty\}$. Furthermore, the following holds.
\begin{align*}
\diam(G-e)
\begin{cases}
\in \{2,3\} &\text{if } w_G(e)=3,\\
= 3  &\text{if } w_G(e)=4,\\
\geq  4 & \text{if } w_G(e)=5,\\
=\infty &\text{if }w_G(e)=\infty.
\end{cases}
\end{align*}
\end{lemma}
\begin{proof}
To see the first statement, suppose that $e$ is an edge with $w_G(e)= k\ge 6$ and let $C=(v_1,\ldots, v_k,v_1)$ with $e = v_1v_2$ be a shortest cycle containing $e$. Since $v_1$ and $v_4$ are non-adjacent and $G$ is of diameter two, $v_1$ and $v_4$ have a common neighbor in $G$, say $z$. Observe that $\{v_1,v_2,v_3,v_4,z\}$ induce a cycle containing $e$ shorter than $C$, a contradiction. Hence, we have $w_G(e)\in \{3,4,5,\infty\}$ for every edge $e$ of $G$.

We prove the second statement. The case when $w_G(e)=3$ is trivial as the distance of any vertex pair increases by at most one by the deletion of $e$.

If $w_G(e)=4$, let $C=(v_1,v_2,v_3,v_4,v_1)$ be a cycle of length four with  $e = v_1v_4$. We claim that $\diam(G-e) = 3$.
Since $w_G(e)>3$, vertices $v_1$ and $v_4$ have no common neighbor in $G$ and hence $\dist_{G-e}(v_1,v_4)= 3$. Consequently $\diam(G-e) \ge 3$.
Suppose for a contradiction that $\diam(G-e) \geq 4$. Then there exists a pair of vertices $x,y\in V(G)$ such that $\dist_{G-e}(x,y)\ge 4> \dist_G(x,y)=2$, and observe that all shortest $x,y$-paths in $G$ must use the deleted edge $v_1v_4$. Since $x$ and $y$ are at distance two in $G$, this is only possible if $|\{x,y\}\cap \{v_1,v_4\}| = 1$, say $x = v_1$ and $y\not\in \{v_1,v_4\}$. Since $\{v_2,y\}\cap \{v_1,v_4\}=\emptyset$, we infer that no shortest $v_2,y$-path in $G$
(which are of length at most two) uses the edge $v_1v_4$, which implies $\dist_{G-e}(v_2,y)= \dist_{G}(v_2,y)$. But then $\dist_{G-e}(x,y)\le \dist_{G-e}(x,v_2)+ \dist_{G-e}(v_2,y) = 1 + \dist_G(v_2,y)\le 3$, a contradiction.

If $w_G(e)=5$ for $e=uv$, then we know that $\dist_{G-e}(u,v) = 4$. Consequently, $\diam(G-e)\ge 4$.
Lastly, an edge with $w_G(e)=\infty$ is a cut-edge, and thus $G-e$ is disconnected. By definition, $\diam(G-e)=\infty$ in this case.
\end{proof}

\begin{lemma}\label{lem:meda3central}
Let $G$ be a graph with at least four vertices having diameter two and girth at most four. Then there exists a set $D\subseteq E(G)$ such that $\diam(G-D)=3$. Furthermore, such a set $D$ of minimum size can be found in polynomial time.
\end{lemma}

\begin{proof}
If there exists an edge $e$ whose deletion results in a graph of diameter three, then $\{e\}$ is an optimal solution.
We can thus make the following.

\begin{assumption}\label{ass:edges-0}
For every edge $e\in E(G)$, we have $\diam(G-e)\neq 3$.
\end{assumption}

Assumption~\ref{ass:edges-0} and \Cref{lem:weight} imply the following.

\begin{claim}\label{ass:edges}
The girth of $G$ is three and for every edge $e\in E(G)$ such that $w_G(e) = 3$, the diameter of $G-e$ is two.
\end{claim}

Furthermore, Assumption~\ref{ass:edges-0} has strong implications on the structure of $G$.

\begin{claim}\label{claim:2conn}
$G$ is biconnected.
\end{claim}

\begin{proof}[Proof of Claim~\ref{claim:2conn}]
Suppose that $v$ is a cut-vertex and observe that $v$ is adjacent with every other vertex of $G$.
Since the girth of $G$ is three, there exists an edge $xy$ such that $v\notin \{x,y\}$. Then $G-vx$ has a vertex pair at distance three, namely $x$ and an arbitrary vertex from a connected component of $G-v$ that does not contain $x$, a contradiction with Assumption~\ref{ass:edges-0}.
\end{proof}

\begin{claim}\label{claim:weight-3}
For every edge $e$ of $G$, we have $w_G(e)= 3$.
\end{claim}

\begin{proof}[Proof of Claim~\ref{claim:weight-3}]
By Claim~\ref{claim:2conn}, all edge weights are finite. By
Assumption~\ref{ass:edges-0} and \Cref{lem:weight}, all edge weights are either $3$ or $5$.
Since the graph has girth three, there is an edge with weight $3$.
Suppose for a contradiction that there exists an edge with weight $5$.
Then, since $G$ is connected, there exists a pair of edges with different weights sharing a vertex, say
$e = uv$ and $f = vz$ where $w_G(e) = 3$ and $w_G(f) = 5$.
The definition of weight function $w$ implies that $\dist_G(u,z) = 2$ and $u-v-z$ is the unique shortest $u,z$-path in $G$.
It follows that the distance in $G-uv$ between $u$ and $z$ is at least three. It is also at most three, since
the edge $uv$ is contained in a triangle in $G$. It follows that $G-uv$ has diameter three,
contradicting Assumption~\ref{ass:edges-0}.
\end{proof}

Claims~\ref{ass:edges} and~\ref{claim:weight-3} imply the following.

\begin{claim}\label{claim:diam-deletion}
For every edge $e$ of $G$, we have $\diam(G-e) = 2$.
\end{claim}

A four-vertex path $P = (a,b,c,d)$ in $G$ is said to be {\it relevant} if (i) the diameter of the graph $G^P = G-(E(G)\cap \{ac,ad,bd\})$ is at most three, and (ii) $G[\{a,b,c,d\}]$ is not an induced cycle with $|N(a)\cap N(d)| = 1$.

\begin{claim}\label{claim:paths}
Let $P = (a,b,c,d)$ be a relevant path in $G$, let $N_G(a) \cap N_G(d) = \{v_1,\ldots, v_k\}$ for some $k\in \mathbb{Z}_{\ge 0}$, and let
$F = \cup_{1\leq i\leq k}\{v_ia,v_id\}$.
Then, there exists an edge set $D\subseteq F\cup (E(G)\cap \{ac,ad,bd\})$ with $|D\cap \{v_ia,v_id\}|\leq 1$ for each $i\in \{1,\ldots, k \}$ such that $G-D$ has diameter three. Furthermore, such an edge set can be computed in polynomial time.
\end{claim}

\begin{proof}[Proof of Claim~\ref{claim:paths}]
If $\diam(G^P)=3$, then we are done (by taking $D = E(G)\cap \{ac,ad,bd\}$). So we assume that $\diam(G^P)=2$.
Note that this implies that $k\ge 1$.
Consider the case when $k=1$. If $\min \{w_{G^P}(v_1a), w_{G^P}(v_1b)\} \leq 4$, then we choose $e_1$ as an edge of minimum weight among $\{v_1a,v_1d\}$. Then $\dist_{G^P-e_1}(a,d)=3$ and it follows that $\diam(G^P-e_1)=3$ by~\Cref{lem:weight}.

\begin{sloppypar}
Suppose $\min \{w_{G^P}(v_1a), w_{G^P}(v_1d)\} \geq 5$,
implying $w_{G^P}(v_1a)= w_{G^P}(v_1d)=5$, since $(a,b,c,d,v_1,a)$ is a cycle of length $5$ in $G^P$ containing the edges $v_1a$ and $v_1d$.
Because of condition (ii) of a relevant path, we may assume that $ac$ is an edge of $G$.
Recall that by~\Cref{claim:weight-3}, vertices $v_1$ and $a$ have a common neighbor in $G$. However,
the fact that $w_{G^P}(v_1a)= w_{G^P}(v_1d)=5$ implies that $v_1b,v_1c \notin E(G)$, and thus we must have $ad\in E(G)$.
This implies that $w_{G-ad}(av_1) = 4$. By Claim~\ref{claim:diam-deletion}, the graph $G-ad$ has diameter two
and therefore by~\Cref{lem:weight} the graph $G-\{ad,av_1\}$ has diameter three. It follows that $\{ad,av_1\}$ is a desired set.

Now we consider the case $k\geq 2$. Notice that $w_{G^P}(e)\leq 4$ for every edge $e$ of the form $v_ia$ or $v_id$. If there exists an edge $f \in F$ with $w_{G^P}(f)=4$, then $D=\{f\}\cup (E(G)\cap \{ac,ad,bd\}$ is a desired edge set.
Therefore, we may assume that $w_{G^P}(e)=3$ for every edge $e \in F$.

Consider an arbitrary edge set $D\subseteq F$ such that $|D\cap \{v_ia,v_id\}|\leq 1$ for each $i$
and \hbox{$\diam(G^P-D)=2$} (for example, $D=\emptyset$). Obviously, there exists a vertex $v_j$ such that $D$ contains neither $v_ja$ nor $v_jd$.
We claim that there exists an edge $e_j\in \{v_ja,v_jd\}$ such that $w_{G^P-D}(e_j)=3$.
Recall that $v_j$ and $a$ must have a common neighbor in $G^P$ by the above assumption.
If $v_j$ and $a$ have a common neighbor in $V(G)\setminus  \{v_1,\ldots, v_k\}$ in the graph $G^P$, then clearly $w_{G^P-D}(v_ja)=3$.
Otherwise %from $w_{G^P}(v_ja)=3$ it follows that
$v_j$ has at least one neighbor $v_{\ell}$ in $G^P$.
From $|D\cap \{v_{\ell}a,v_{\ell}d\}|\leq 1$, at least one of $\{a,v_j,v_{\ell}\}$ and $\{d,v_j,v_{\ell}\}$ forms a triangle in $G^P-D$,
thus ensuring the existence of $e_j\in \{v_ja,v_jd\}$ such that $w_{G^P-D}(e_j)=3$.
\end{sloppypar}

We now define a sequence of sets $D_0,\ldots, D_k$
such that for all $j\in \{0,1,\ldots, k\}$ we have $D_j\subseteq F$, $|D_j\cap \{v_ia,v_id\}|\leq 1$ for each $i$,
\hbox{$\diam(G^P-D_j)=2$} for all $j<k$, and \hbox{$\diam(G^P-D_k)=3$}.
We set $D_0 = \emptyset$ and for each $j=1,\ldots, k$, we set $D_j = D_{j-1}\cup\{e_j\}$ where $v_i$ is a vertex such that $D_{j-1}$ contains neither $v_ia$ nor $v_id$ and $e_j$ is an edge in $\{v_ia,v_id\}$ such that $w_{G^P-D_{j-1}}(e_j)=3$. Note that such an edge exists by the claim proved in the previous paragraph.
Clearly, the final set $D_k$ can be computed in polynomial time and satisfies $|D_k\cap \{v_ia,v_id\}|= 1$ for each $i$ and $\diam(G^P-D_k)\ge \dist_{G^P-D_k}(a,d)= 3$. Thus, $D_k\cup (E(G)\cap \{ac,ad,bd\})$ is a desired set. This proves Claim~\ref{claim:paths}.
\end{proof}

For a relevant path $Q=(a,b,c,d)$, let us denote by $f(Q)$ the quantity given by the following expression:
 $$f(Q) = |E(G)\cap\{ad,bc,bd\}|+|N_G(a)\cap N_G(d)|\,.$$
By Claim~\ref{claim:paths}, for every relevant path $Q=(a,b,c,d)$ one can compute an edge set $D$ of size at most $f(Q)$ such that $G-D$ has diameter three.
Therefore, to see the existence of an edge set whose deletion results in a graph with diameter three, it suffices to show that there exists a relevant path $Q$.

\begin{claim}\label{claim:relevant}
$G$ has a relevant path.
\end{claim}

\begin{proof}
Let $P = (a,b,c,d)$ be an arbitrary four-vertex path in $G$ and suppose that $P$ is not relevant. Since every induced $P_4$ is relevant, $P$ cannot be induced, that is, $E(G)\cap \{ac,ad,bd\}\neq \emptyset$, or, equivalently, $G^P\neq G$.

Suppose first that $E(G)\cap \{ac,ad,bd\} = \{ad\}$. In this case, $G[\{a,b,c,d\}]$ is an induced cycle, $w_G(ad)\in \{3,4\}$, and \Cref{lem:weight} implies that $\diam(G^P) = \diam(G-ad) \le 3$. Since $P$ is not relevant, we have $|N(a)\cap N(d)| = 1$. Let $v$ be the unique common neighbor of $a$ and $d$, and consider the two four-vertex paths $Q = (v,a,b,c)$ and $R = (v,d,c,b)$. If one of $R$ and $Q$ is an induced $P_4$, then it is a relevant path. So we may assume that $bv\in E(G)$ or $cv\in E(G)$. By symmetry, we may assume that $bv\in E(G)$. We will show that in this case $Q$ is a relevant path. Since $G[\{v,a,b,c\}]$ is not an induced $C_4$, it suffices to show that
 $\diam(G^Q)\le 3$. Note first that by Claim~\ref{claim:diam-deletion}, the graph $G-bv$ has diameter two.
If $cv\not\in E(G)$, then $G^Q = G-bv$ and we are done. Suppose that $cv\in E(G)$. Then $G^Q = G-\{bv,cv\}$. Since $d$ is a common neighbor in $G-bv$ of $c$ and $v$, we have $w_{G-bv}(cv) = 3$ and we infer using \Cref{lem:weight} that $\diam(G^Q) = \diam(G-\{bc,cv\})\le 3$. Thus, in both cases path $Q$ is relevant.

The above analysis implies that we may assume that $G$ is a $\{P_4,C_4\}$-free graph. Since $G$ is connected, this implies that $G$ has a universal vertex, see~\cite{Wolk}. Let $U$ be the set of all universal vertices in $G$. Since $G$ is not a complete graph, $U\neq V(G)$. Moreover, since the graph $G-U$ is $\{P_4,C_4\}$-free and without universal vertices, it has at least two connected components. Consequently $|U|\ge 2$ by \Cref{claim:2conn}.
Consider now the path $Q = (a,b,c,d)$ where $a$ and $d$ are vertices from different connected components of $G-U$ and $b,c\in U$.
Since $b$ and $c$ are universal vertices in $G$, we have that $ac,bd\in E(G)$.  Claim~\ref{claim:diam-deletion} implies that $\diam(G-ac) = 2$.
Moreover, $w_{G-ac}(bd) = 3$ and \Cref{lem:weight} implies that $\diam(G-\{ac,bd\}) \le 3$.
Thus, $\diam(G^Q) \le 3$ and hence $Q$ is a relevant path in $G$.
This completes the proof.
\end{proof}

By Claims~\ref{claim:paths} and~\ref{claim:relevant}, we conclude that there exists an edge set $D$ such that $\diam(G-D)=3$.
On the other hand, for every relevant path $Q$, any solution $D$ which makes  $Q$  a diametral path contains at least $f(Q)$ edges.
Furthermore, any path $P^*$ achieving diameter three after removing an optimal solution $D^*$ is relevant, and an output solution of size at most $f(P^*)$ provided by Claim~\ref{claim:paths} meets the lower bound and
thus is an optimal solution (even though it need not be necessarily identical with $D^*$). These observations bring forth the following algorithm: we iterate over all four-vertex paths $Q = (a,b,c,d)$ in $G$ and compute a solution of size at most $f(Q)$ whenever $Q$ is relevant.
By keeping track of a thus-far best solution, we can compute an optimal solution. Clearly, the algorithm runs in polynomial time.
\end{proof}

\begin{theorem}\label{thm:spanning-subgraph-diameter-3}
Let $G$ be a graph.
Then, $G$ has a spanning subgraph of diameter three if and only if one of the following conditions holds.
\begin{enumerate}[(i)]
  \item\label{item-complete} $G$ has diameter one (that is, $G$ is a complete graph) and at least four vertices.
  \item\label{item-diameter-two} $G$ has diameter two and girth at most four.
  \item\label{item-diameter-three} $G$ has diameter three.
\end{enumerate}
\end{theorem}

\begin{proof}
Suppose first that $G$ has a spanning subgraph $G'$ such that $\diam(G') = 3$. Since $\diam(G')\ge \diam(G)$, we infer that $G$ has diameter at most three. Furthermore, since $G'$ contains a pair of vertices at distance three, $G'$ (and therefore $G$) has at least four vertices.
If $G$ has diameter one or three, then one of conditions~\eqref{item-complete} and~\eqref{item-diameter-three} holds.

Let us assume now that $\diam(G) = 2$. It remains to show that $G$ has girth at most four.
Suppose that the girth of $G$ is at least five. By~\Cref{lem:weight}, we have $w_G(e)\in \{5,\infty\}$ for every edge $e$ of $G$. Then~\Cref{lem:weight} implies that $G$ does not contain any spanning subgraph of diameter three, a contradiction.
This establishes the forward implication.

\medskip
We now prove that the conditions are also sufficient.
Let $G$ be a graph satisfying one of conditions~\eqref{item-complete}--\eqref{item-diameter-three}. We will show that in each case, $G$ contains a spanning subgraph $G'$ with diameter three.
If condition~\eqref{item-complete} holds, observe that for every edge $e\in E(G)$, the graph $G-e$ has diameter two and girth three, which reduces to the case of~\eqref{item-diameter-two}. If condition~\eqref{item-diameter-two} holds, that is, $G$ has diameter two and girth at most four, then a desired subgraph $G'$ exists by~\Cref{lem:meda3central}.
If condition~\eqref{item-diameter-three} holds, then we can take $G' = G$.
This completes the proof.
\end{proof}

Theorem~\ref{thm:spanning-subgraph-diameter-3} implies the following.

\begin{corollary}
The \eda{3} problem is solvable in polynomial time.
\end{corollary}

In case of a yes instance, a smallest set of edges whose deletion results in a graph of diameter $3$ can also be computed
efficiently.

\begin{theorem}\label{thm:meda3poly}
The \meda{3} problem is solvable in polynomial time.
\end{theorem}
\begin{proof}
If $G$ has diameter three, the solution is trivial (delete nothing). If $G$ has at most three vertices or has diameter at least four, then it is a trivial no instance. We may also assume that $G$ is not complete. Indeed, if $G$ is complete, then we can delete any edge to transform $G$ into an equivalent graph $G'$ of diameter two and girth three. Solving the resulting instance $G'$ optimally will yield an optimal solution for $G$.
Clearly, whether $G$ satisfies one of the above conditions (and if so, which one) can be checked in polynomial time.
Therefore, we may assume that $G$ has at least four vertices and diameter two.
Moreover, we may assume that $G$ has girth at most four, since otherwise $G$ is trivially a no instance  by \Cref{lem:weight}.
Now, applying~\Cref{lem:meda3central} we can compute an optimal solution.
\end{proof}

The result of Theorem~\ref{thm:spanning-subgraph-diameter-3} connects our study to a well-known concept in extremal and algebraic graph theory, the Moore graphs. A connected graph with diameter $d$ is either acyclic or has girth at most $2d+1$.
A connected graph $G$ with diameter $d$ and girth exactly $2d+1$ is said to be a \emph{Moore graph} (see~\cite{S68}).
A Moore graph is necessarily regular (that is, all vertices have the same degree), and, apart from odd cycles and complete graphs, all Moore graphs have diameter two. Furthermore, it is known that every Moore graph $G$ with diameter two is $k$-regular with $|V(G)| = k^2+1$ for some $k\in \{2,3,7,57\}$. For each $k \in \{2,3,7\}$, there is a unique $k$-regular Moore graph with diameter two, namely the $5$-cycle (for $k = 2$), the ($10$-vertex) Petersen graph (for $k = 3$), and the
($50$-vertex) Hoffman-Singleton graph (for $k = 7$); see, e.g.,~\cite{BCN89}. The existence of a ($3250$-vertex) $57$-regular Moore graph is a famous open problem in graph theory (see~\cite{D19} for a recent survey).

Theorem~\ref{thm:spanning-subgraph-diameter-3} implies the following characterization of
graphs with diameter two that are yes / no instances to the \eda{3} problem.

\begin{corollary}
For every graph $G$ with diameter two, the following holds.
\begin{enumerate}
  \item If $G$ has girth at most four, then $G$ has a spanning subgraph of diameter three.
  \item If $G$ has girth at least five, then $G$ is either a Moore graph or $G$ is acyclic (in which case $G$ is isomorphic to a star $K_{1,n}$ for some $n\ge 2$); in either case, $G$ does not have any spanning subgraph of diameter three.
\end{enumerate}
\end{corollary}

\section{{\sf NP}-completeness results for \meda{d}, $d\ge 5$}\label{5meda}

We prove that \meda{5} is {\sf NP}-complete for graphs with diameter three and four.
Then we generalize this result to \meda{d} for all $d\ge 6$.

\begin{theorem}\label{3to5}
The \meda{5} problem is {\sf NP}-complete  for graphs of diameter $3$.
\end{theorem}

\begin{proof}
Since computing the diameter of a graph can be done in polynomial time, \meda{5} is in {\sf NP}.

Our reduction is from the {\sf NP}-complete \textsc{Vertex Cover} problem (see~\cite{GJ79}).

\begin{center}
\fbox{\parbox{0.95\linewidth}{\noindent
\textsc{Vertex Cover} \\[.8ex]
\begin{tabular*}{.93\textwidth}{rl}
{\em Input:} & A  graph $\Gamma=(W,E')$, an integer $c$. \\
{\em Question:} & Does $\Gamma$ have a vertex cover of size at most $c$\,?
\end{tabular*}
}}
\end{center}

Given an instance $(\Gamma=(W,E'),c)$ of \textsc{Vertex Cover} we create an instance $(G,k)$ of the
\meda{5} problem such that $G=(V,E)$ is of diameter~$3$ as follows.

We set $k=\vert W\vert+c$.

For each vertex $v\in W$ we create a path $v_1-v_2-v_3-v_4$.

For each edge $e=uv\in E'$ we create two paths $v_2-e_1-u_3$ and $v_3-e_2-u_2$.

There is a vertex $s$ with the edges $sv_1$ and $sv_2$ for every $v\in W$.

There is a vertex $t$ with the edges $tv_3$ and $tv_4$ for every $v\in W$.

Finally, there are  $K^1$ and $K^2$, two complete graphs each with $k+1$ vertices. For every vertex $u\in K^1$ and each  $v\in W$ there are the edges $uv_1$ ad $uv_2$, and there is an edge $ue_1$ for every $e\in E'$. For every vertex $u\in K^2$ and each $v\in W$ there are the edges $uv_3$ and $uv_4$, and there is an edge $ue_2$ for every $e\in E'$. There is an edge between any pair of vertices of $K^1\cup K^2$.

There are no other vertices or edges in $G$.

Clearly $(G,k)$ can be computed from $(\Gamma,c)$ in polynomial time.
Furthermore, it can be verified that $G$ has diameter $3$.

\medskip
Suppose that $(\Gamma,c)$ is a yes instance of \textsc{Vertex Cover}.
Let $S\subseteq W$ be a vertex cover in $\Gamma$ of size $\vert S\vert\le c$. We build $F\subseteq E$ as follows:
for each $v\in  S$ we put the edges $sv_2$ and $tv_3$ in $F$; for each $v\in W\setminus S$ we put edge $v_2v_3$ in $F$.
Thus $\vert F\vert\le k$. We show that $\dist_{G-F}(s,t)=5$. First, $s-v_1-z_1-z_2-v_4-t$ where $v\in W$, $z_1\in K^1$ and $z_2\in K^2$, is a path of length $5$ in $G-F$. Now let $P$ be a shortest $s,t$-path in $G-F$. Since $\dist_G(s,t)=3$ the length of $P$ is at least $3$. If $P$ has length less than $5$ it cannot contain a vertex of $K^1\cup K^2$. Assume that the length of $P$ is $3$.  Then $P=s-v_2-v_3-t$ for some $v\in W$. If $v\in S$ then $sv_2\in F$ and $P$ cannot be a path in $G-F$. If $v\not\in S$ then $v_2v_3\in F$ and $P$ cannot be a path in $G-F$. So $P$ has length at least $4$.
Assume that the length of $P$ is $4$. Suppose that the second vertex in $P$ is $v_1$ for some $v\in W$.
Then $P=s-v_1-v_2-v_3-t$, which is impossible since either $v\in S$ and $v_3t\in F$ or $v\not\in S$ and $v_2v_3\in F$.
So the second vertex of $P$ is $v_2$ for some $v\in W$. Therefore $sv_2\not \in F$, which implies that $v\not \in S$, and thus $v_2v_3\in F$.
Thus $P=s-v_2-e_i-w_3-t$ for some $e= vw\in E'$ and some $i\in \{1,2\}$. Since $v\not\in S$ and $S$ is a vertex cover of $\Gamma$, we have $w\in S$.
But $u_3t\in F$, a contradiction. We conclude that the length of $P$ is $5$. Now, for any pair of vertices $\{x,y\}\ne \{s,t\}$ we have $\dist_{G-F}(x,y)\le 4$ (passing through  the vertices of $K^1\cup K^2$). Thus any shortest path that is not a $s,t$-path has a length no more than $4$. Hence the diameter of $G-F$ is $5$ and $(G,k)$ is a yes instance of \meda{5}.

\medskip
Suppose that $(G,k)$ is a yes instance of \meda{5}. So there exists a set $F\subseteq E$, $\vert F\vert\le k$, such that the diameter of $G-F$ is $5$. It can be verified that for any two distinct vertices $x,y\in V(G)$ such that
$\{x,y\}\ne \{s,t\}$, there are at least $k+1$ edge-disjoint $x,y$-paths of length at most four in $G$ (using the vertices of $K^1\cup K^2$).
Therefore, $\dist_{G-F}(x,y)\le 4$. This implies that a shortest path of length $5$ in $G-F$ is an $s,t$-path. Note that each $s,t$-path with a vertex in $K^1\cup K^2$ has length at least five. We may thus assume that $F$ does not contain any edge incident to a vertex of $K^1\cup K^2$.

The only $s,t$-paths of length 3 in $G$ are of the form $s-v_2-v_3-t$ for $v\in W$. They are pairwise edge-disjoint, so at least $\vert W\vert$ of their edges are in $F$, at least one per path. The $s,t$-paths of length 4 in $G$ are paths $s-v_1-v_2-v_3-t$
or $s-v_2-v_3-v_4-t$ for $v\in W$, and paths $s-v_2-e_i-w_3-t$ for $e = vw\in E'$ and $i\in \{1,2\}$.
All the other $s,t$-paths of $G$ have length more than 4.

Suppose that $F$ contains an edge of the form $v_2e_i$. This edge cannot be in an $s,t$-path of length $3$. Since the path $s-v_2-e_i-u_3-t$ is a unique $s,t$-path of length $4$ that contains this edge, we can remove $v_2e_i$ from $F$ and add instead the edge $sv_2$ (if $sv_2$ is not already in $F$). So from now, by symmetry, we may assume that for every $v\in W$, the edges $v_2e_i$ and $e_iv_3$ are not in $F$.

Suppose that $F$ contains an edge of the form $v_1v_2$. Since there is exactly one $s,t$-path of length $4$ that contains this edge, the path $s-v_1-v_2-v_3-t$, and the edge $v_1v_2$ cannot be in an $s,t$-path of length $3$, we can remove $v_1v_2$ from $F$ and add instead the edge $v_2v_3$ (if $v_2v_3$ is not already in $F$). The same argument can be used with the edge $sv_1$. So from now, by symmetry, we may assume that each edge in $F$ is of the form
$sv_2$, $v_2v_3$, or $v_3t$ for some $v\in W$.

Since for every $v\in W$, set $F$ contains at least one of the edges of the path $s-v_1-v_2-v_3-t$, we infer that
$F$ contains at least one of the edges $v_2v_3$ and $v_3t$, for every $v\in W$. By symmetry, $F$ also contains at least one of the edges $sv_2$ and $v_2v_3$ for every $v\in W$. Thus, if for some $v\in W$ we have $v_2v_3\not \in F$, then $sv_2\in F$ and $v_3t\in F$.

Furthermore, if for some $v\in W$ we have $sv_2\in F$ and $v_3t\in F$, then we may assume that $v_2v_3\not \in F$.
Indeed, if $\{sv_2,v_2v_3,v_3t\}\subseteq F$ for some $v\in W$, then we can replace $F$ with $F\setminus\{v_2v_3\}$:
adding the edge $v_2v_3$ to the graph $G-F$ cannot create an $s,t$-path of length 4 or less since $sv_2\in F$ and $v_3t\in F$.

From the above we infer that $F$ has the following structure. For each $v\in W$ either $v_2v_3\in F$ or $sv_2,v_3t\in F$. Since $\vert F\vert\le k=\vert W\vert+c$, at most $c$ pairs $sv_2,v_3t$ are in $F$. From $F$ we build a set $S\subseteq W$ sas follows. We take $v\in S$ if and only if $sv_2,v_3t\in F$.
So $\vert S\vert \le c$. For each edge $uv\in E'$ we have $sv_2,v_3t\in F$ or $su_2,u_3t\in F$ else $s-v_2-e_i-u_3-t$ for some $i\in \{1,2\}$
would be an $s,t$-path in $G-F$ of length $4$. Thus $S$ is a vertex cover of $\Gamma$ and $(\Gamma,c)$ is a yes instance.
\end{proof}

\begin{theorem}\label{4to5}
The \meda{5} problem is {\sf NP}-complete for graphs of diameter $4$.
\end{theorem}

\begin{proof}
The proof is the same as the proof of Theorem~\ref{3to5} except for the following.
Instead of two complete subgraphs $K^1,K^2$, here there are four  pairwise disjoint complete subgraphs $K^1,K^2,K^3,K^4$ with $k+1$ vertices each.
All the vertices of $K^3$ are adjacent to all the vertices of $K^1\cup K^4$.
All the vertices of $K^4$ are adjacent to all the vertices of $K^2\cup K^3$.
The vertices of $K^1$ and $K^2$ are linked to the rest of $G$ as previously.
The so obtained graph now has diameter $4$ (the diametral paths are from $u\in K^3$ to $t$ and from $v\in K^4$ to $s$).
\end{proof}

\begin{lemma}\label{dtod+1}
For any fixed $d\ge 5$ the
\meda{(d+1)} problem such that $G=(V,E)$ is of diameter~$d$ is {\sf NP}-complete.
\end{lemma}

\begin{proof}
Given an instance $(G,p)$ of the \meda{5} problem such that $G=(V,E)$ is of diameter~$4$ we create an instance $(G',p')$ of \meda{(d+1)} problem such that $G'=(V',E')$ is of diameter $d$ as follows. We take $p=p'$. Let $\{s,t\}$ be a diametral pair of $G$. Let $d'=d-4$. We add to $G$ a path with (new) vertices
$q_1,q_2,\ldots,q_{d'}$ and an edge $tq_1$. Let $G=(V,E)$ be a yes instance of \meda{5}. Let  $F\subseteq E, \vert F\vert\le p$ be such that $G-F$ has diameter $5$. Then $G'-F$ has diameter $5+d' = d+1$. Now, let $G'=(V,'E')$ be a yes instance of \meda{(d+1)} and  $F\subseteq E', \vert F\vert\le p'=p$ be such that $G'-F$ has diameter $d+1$.
Since $G'-F$ is connected, $F\subseteq E$ and $G-F$ is connected. It follows from the proof of Theorem~\ref{3to5} (which is the basis for the proof of Lemma~\ref{4to5}) that for every two vertices $x,y$ in $G$ such that $\{x,y\}\neq \{s,t\}$, we have
$\dist_{G'-F}(x,y)= \dist_{G-F}(x,y)\le 4$. This implies that vertices $s$ and $q_{d'}$ form the only diametral pair of $G'$ and thus $G-F$ has diameter $d+1-d'=d+1-(d-4)=5$.
\end{proof}

Clearly, Theorem~\ref{4to5} and Lemma~\ref{dtod+1} have the following consequence.

\begin{corollary}\label{d-meda}
For every $d\ge 5$, the \meda{d} problem is {\sf NP}-complete.
\end{corollary}

We now show how to further adapt the above constructions to prove {\sf NP}-completeness for \meda{d} problems, $d\ge 5$, also with the additional restriction that the aim is to increase the diameter of the input graph for a fixed value $k\ge 2$. The possible values of $k$ for which we obtain {\sf NP}-completeness depend on the diameter $d$
of the target graph (or, equivalently, on the diameter of the input graph); see Theorem~\ref{dtod+k} for the precise statement.

\begin{lemma}\label{numberedge}
For any fixed $d\ge 4$ and $\delta\ge 1$ there are {\sf NP}-hard instances $(G,p)$ of the \meda{(d+1)} problem such that the following holds.
There is no set $F\subseteq E$ with fewer than $p$ edges such that $G-F$ has diameter $d+1$ and there is no edge set $F'\subseteq E,\vert F'\vert\le p+\delta,$ such that $G-F'$ has a diameter at least $d+2$.
\end{lemma}

\begin{proof}
We showed in the proof of Theorem~\ref{3to5} that there are {\sf NP}-hard instances $(G,p)$ of \meda{5} for a graph $G$ with a diameter $3$ such that $p$ edges are necessary (that is, there is no set $F\subseteq E$ with fewer than $p$ edges such that $G-F$ has diameter $5$).
Hence through the proof of Theorem~\ref{4to5}  the same holds for an instance $(G,p)$ of \meda{5} when $G$ has a diameter $4$. From a such instance $(G,p)$ we build $(G_{\delta},p_{\delta})$ another instance of \meda{5} where $G_\delta$ has a diameter $4$. We take $p_{\delta}=(\delta+1) p$. We take $\delta+1$ copies of the subgraph of $G$ induced by the all vertices $v_i$ and $e_j$. Here the four cliques $K^1,K^2,K^3, K^4$ have size $\delta+1$ times the size they have in the proof of Theorem~\ref{4to5}. The copies of the vertices $v_i,e_j$ are connected in the same manner as above to the four cliques $K^1,K^2,K^3,K^4$. There are two vertices $s,t$ connected to the copies of the vertices $v_j$ as before. Hence the $\delta+1$ copies of $G$ are independent from each other. Note that $G_\delta$ has diameter $4$. As noted in the proof of Theorem~\ref{3to5}, only the edges between the vertices $s,t,v_i,e_j$ can be put inside the edge set $F$ that increase the diameter of the graph by one unit. Thus there is no set $F\subseteq E$ with fewer than $p'=(\delta+1)p$ edges such that $G_\delta-F$ has diameter $5$. Furthermore, there is no edge set $F'$ with at most $p'+\delta$ edges such that $G_\delta-F$ has diameter more than $5$, since the $\delta+1$ copies of $G$ are independent from each other. From the construction used in the proof of Lemma \ref{dtod+1} the same holds for an instance $(G,p)$ of \meda{(d+1)} when $G$ has diameter $d$.
\end{proof}

\begin{lemma}\label{dtod+k+1}
For every fixed $d$ and $k$, $d\ge 4,k\ge 1$, the \meda{(d+2k+1)} problem is {\sf NP}-complete  for graphs of diameter $d+k$.
\end{lemma}

\begin{proof}
Given an instance $(G,p)$ of the
\meda{(d+1)} problem such that $G=(V,E)$ is of diameter~$d$ and satisfies the condition of Lemma \ref{numberedge} with $\delta = k$, we create an instance $(G',p')$ of \meda{(d+2k+1)} problem such that $G'=(V',E')$ is of diameter $d+k$ as follows. Let $p'=p+k$.
Let $\{s,t\}$ be a diametral pair of $G$. We add to $G$ a path with (new) vertices
$q_1,q_2,\ldots,q_{k}$ and an edge $tq_1$. We denote $t=q_0$. For each $i\in \{1,\ldots, k\}$ we add a new vertex $r_i$ with the two edges $q_{i-1}r_i$, $r_iq_i$. Let us denote by $T_i$ the edge set of the triangle containing $r_i$. Clearly $G'$ has a diameter $d+k$.  Let $G=(V,E)$ be a yes instance of \meda{(d+1)}. Let  $F\subseteq E, \vert F\vert\le p$ be such that $G-F$ has diameter $d+1$.
Let $F'=F \cup \{q_0q_1\}\cup\ldots\cup \{q_{k-1}q_k\}$. Then $\vert F'\vert= p+k\le p'$ and $G'-F'$ has diameter $d+1+2k$.
Now, let $G'=(V,'E')$ be a yes instance of \meda{(d+2k+1)} and $F'\subseteq E', \vert F'\vert\le p'$ be such that $G'-F$ has diameter $d+2k+1$.  From Lemma \ref{numberedge} the diameter $d$ of $G$ cannot be increased to $d+\delta$ where $\delta\ge 2,$ by deleting only $p+k=p'$ edges.
The graph $G'-F'$ is connected, so $\vert F'\cap T_i\vert\le 1$ whenever $1\le i\le k$. This implies that $F'\cap T_i=\{q_{i-1}q_i \},1\le  i\le k$ and $\vert F'\cap E\vert\le p$. Thus $F=F'\cap E$ is such that $G-F$ has diameter $d+1$.
\end{proof}

\begin{lemma}\label{dtod+k+2}
For every fixed $d$ and $k$, $d\ge 3,k\ge 1$, with $d+k\ge 5$ the \meda{(d+2k+2)} problem is {\sf NP}-complete  for graphs of diameter $d+k$.
\end{lemma}

\begin{proof}
From Theorem~\ref{3to5} we know that the \meda{5} problem is {\sf NP}-complete for graphs of diameter $3$. Using similar arguments as in the proof of Lemma \ref{dtod+1} the \meda{(d+2)} problem is {\sf NP}-complete  for graphs of diameter $d$. Now with similar arguments as those used in the proofs of Lemmas \ref{numberedge} and \ref{dtod+k+1} the proof is complete.
\end{proof}

Using Lemmas~\ref{dtod+1}, \ref{dtod+k+1}, and~\ref{dtod+k+2} we now derive the following general result.

\begin{theorem}\label{dtod+k}
For every fixed $d\ge 5$ and $k$ such that $1\le k\le d-1$, the \meda{(d+k)} problem is {\sf NP}-complete for graphs of diameter $d$.
\end{theorem}

\begin{proof}
If $k=1$, then we use Lemma \ref{dtod+1}.

For $k= 2$, let $d'=d-1$ and $k' = 1$. Then $d'\ge 4$ and applying Lemma \ref{dtod+k+1} with $d'$ and $k'$ in place of $d$ and $k$, respectively, we have that \meda{(d' + 3)} is {\sf NP}-complete  for graphs of diameter $d'+1$. Since $d=d'+1$ and $d+2=d'+3$ the result follows.

For $k\ge 3$, let $d'=d-k+2$ and $k'=k-2$. Since $d\ge k+1$, we have $d'\ge 3$.
Since $k\ge 3$ and $d\ge 5$, we also have $k'\ge 1$ and $d'+k' \ge 5$.
 From Lemma \ref{dtod+k+2} we have that \meda{(d' + 2k' + 2)} is {\sf NP}-complete for graphs of diameter $d'+k'$.
 Since  $d=d'+k'$ and $d+k=d'+2k'+2$, the result follows.
\end{proof}

\section{Conclusion and future works}\label{concl}

We determined the computational complexity of several problems related to the blocker problem for the diameter of a graph.
In particular, we summarize in Table~\ref{table:summary} the complexity results for the
\meda{(d+k)} problem when restricted to graphs of diameter $d$ for various values of $d\ge 1$ and $k\ge 1$.

\begin{table}[h!]
\begin{center}
\footnotesize
  \begin{tabular}{|c|c|c|c|c|c|c|c|}
\hline
$d~\backslash~k$ & $1$ & $2$ & $3$ & $4$ & $5$ & $ 6$ & $7$\\
\hline
    $1$ & {\sf P} (T~\ref{thm:complete}) & {\sf P} (T~\ref{thm:complete}) & {\sf P} (T~\ref{thm:complete}) & {\sf P} (T~\ref{thm:complete})
        &{\sf P} (T~\ref{thm:complete}) & {\sf P} (T~\ref{thm:complete}) & {\sf P} (T~\ref{thm:complete}) \\
        \hline
    $2$ & {\sf P} (T~\ref{thm:meda3poly}) & ? & ? & ? & ? & ? & ? \\
        \hline
    $3$ & ? & {\sf NP}-c (T~\ref{3to5}) & ? & ? & ? & ? & ? \\
    \hline
    $4$ &  {\sf NP}-c (T~\ref{4to5}) & ? & ? & ? & ? & ? & ? \\
    \hline
    $5$ &  {\sf NP}-c (T~\ref{dtod+k}) & {\sf NP}-c (T~\ref{dtod+k}) & {\sf NP}-c (T~\ref{dtod+k}) & {\sf NP}-c (T~\ref{dtod+k}) & ? & ?  & ? \\
    \hline
    $6$ &  {\sf NP}-c (T~\ref{dtod+k}) & {\sf NP}-c (T~\ref{dtod+k}) & {\sf NP}-c (T~\ref{dtod+k}) & {\sf NP}-c (T~\ref{dtod+k}) & {\sf NP}-c (T~\ref{dtod+k}) & ? & ? \\
    \hline
    $7$ &  {\sf NP}-c  (T~\ref{dtod+k}) & {\sf NP}-c (T~\ref{dtod+k}) & {\sf NP}-c (T~\ref{dtod+k}) &{\sf NP}-c (T~\ref{dtod+k}) & {\sf NP}-c (T~\ref{dtod+k})& {\sf NP}-c (T~\ref{dtod+k})  & ? \\
    \hline
    $8$ &  {\sf NP}-c  (T~\ref{dtod+k}) & {\sf NP}-c (T~\ref{dtod+k}) & {\sf NP}-c (T~\ref{dtod+k}) &{\sf NP}-c (T~\ref{dtod+k}) & {\sf NP}-c (T~\ref{dtod+k}) &{\sf NP}-c (T~\ref{dtod+k})  &{\sf NP}-c (T~\ref{dtod+k}) \\
    \hline
\end{tabular}
\end{center}
\caption{Complexity of the \medalong{(d+k)} problem when restricted to graphs of diameter $d$, for various values of $d\ge 1$ and $k\ge 1$.
Rows are indexed by the value $d$ of the diameter of the input graph. Columns are indexed by the desired value $k$ of the diameter increase. {\sf P} denotes that the problem is solvable in polynomial time,
{\sf NP}-c that it is {\sf NP}-complete, and a question mark that the complexity is open. Next to each entry, we give the reference proving the corresponding statement (T stands for Theorem).
\label{table:summary}}
\end{table}

In conclusion, we list some related open problems and research directions that seem interesting for future consideration.
The \meda{d} problem is polynomial for $d = 3$ and {\sf NP}-complete for all $d\ge 5$. Thus, to obtain a dichotomy, one open problem remains, the complexity of \meda{4}. It would be natural to also investigate the \eda{d} and \mda{d} problems for $d\ge 4$. In the cases
of {\sf NP}-complete constant diameter augmentation problems, it would be interesting to understand
restrictions on the input graphs under which the problems become tractable, and perform a more detailed complexity analysis of the problems, for example with respect to parameterized complexity, existence of polynomial kernels, and approximability properties of the corresponding
optimization problems.

From the structural point of view, an interesting question related to the \edalong{d} problems is the following. What are the graphs $G$ such that the set of diameter values of the spanning connected subgraphs of $G$ forms an interval of consecutive integers? This class of graphs generalizes trees and complete graphs. Examples of graphs not in the class include Moore graphs of diameter two (for example, the Petersen graph).

\subsection*{Acknowledgements}

The authors are grateful to Nina Chiarelli and Ademir Hujdurovi\'c for fruitful discussions. The second named author is supported in part by the Slovenian Research Agency (I0-0035, research program P1-0285, and research projects J1-9110, N1-0102, and N1-0160). Part of this work was done while the second named author was visiting LAMSADE, University Paris-Dauphine; their support and hospitality is gratefully acknowledged.


\begin{thebibliography}{99}
\bibitem{BEHKKPSS10}
G. Baier, T. Erlebach, A. Hall, E. K\"{o}hler, P. Kolman, O. Pangr\'ac, H. Schilling, M. Skutella, Length-bounded cuts and flows,
ACM Transactions on Algorithms 7 (2010) Art. 4, 27 pp.

\bibitem{BBPR}
C. Bazgan, C. Bentz, C. Picouleau and B. Ries,
Blockers for the stability number and the chromatic number,
Graphs and Combinatorics 31 (2015) 73--90.

\bibitem{BTT11}
C. Bazgan, S. Toubaline and Z. Tuza,
The most vital nodes with respect to independent set and vertex cover,
Discrete Applied Mathematics 159 (2011) 1933--1946.

\bibitem{Bentz}
C. Bentz, M.-C. Costa, D. de Werra, C. Picouleau and B. Ries,
Weighted transversals and blockers for some optimization problems in graphs,
In: Progress in Combinatorial Optimization, pp.~203--222, Wiley-ISTE, 2012.

\bibitem{BBKP18}
K. B\'erczi, A. Bern\'ath, T. Kir\'aly, Gy. Pap,
Blocking optimal structures,
Discrete Mathematics 341 (2018) 1864--1872.

\bibitem{BCN89}
A. E. Brouwer, A. M.  Cohen, A. Neumaier, Distance-regular graphs. Springer-Verlag, Berlin, 1989.

\bibitem{Chung87}
F. R. K. Chung, Diameters of graphs: old problems and new results. Eighteenth Southeastern International Conference on Combinatorics, Graph Theory, and Computing (Boca Raton, Fla., 1987). Congressus Numerantium 60 (1987) 295--317.

\bibitem{Chung}
F. R. K. Chung, Graphs with small diameter after edge deletion, Discrete Applied Mathematics 37/38 (1992) 73--94.


\bibitem{ChungGarey}
F. R. K. Chung, M. R. Garey, Diameter bounds for altered graphs, Journal of Graph Theory 8 (1984) 511--534.

\bibitem{CWP11}
M.-C. Costa, D. de Werra, C. Picouleau, Minimum $d$-blockers and $d$-transversals in graphs, Journal of Combinatorial Optimization 22 (2011) 857--872.

\bibitem{D19}
C. Dalf\'o, A survey on the missing Moore graph, Linear Algebra and Applications 569 (2019) 1--14.

\bibitem{DPPR15}
\"O. Diner, D. Paulusma, C. Picouleau and B. Ries, Contraction and deletion blockers for perfect graphs and $H$-free graphs, Theoretical Computer Science, 746 (2018) 49--72.

\bibitem{FGGM15}
F. Frati, S. Gaspers, J. Gudmundsson, L. Mathieson, Augmenting graphs to minimize the diameter, Algorithmica 72
(2015) 995--1010.

\bibitem{GPR19}
E. Galby, P.T. Lima and B. Ries, Reducing the domination number of graphs via edge contractions. In Proc. of the 44th International Symposium on Mathematical Foundations of Computer Science (MFCS), volume 138 of LIPIcs, pages 41:1--41:13, 2019.

\bibitem{GHN13}
Y. Gao, D. R. Hare, J. Nastos, The parametric complexity of graph diameter augmentation,
 Discrete Applied Mathematics 161 (2013) 1626--1631.

\bibitem{GJ79}
M. R. Garey and D. S. Johnson,  Computers and Intractability: A Guide to the Theory of NP-Completeness, Freeman, 1979.

%\bibitem{GR01}
%C. Godsil, G. Royle, Algebraic graph theory. Springer-Verlag, New York, 2001.

%\bibitem{HHS82}
%P. L. Hammer, P. Hansen, B. Simeone, Vertices belonging to all or to no maximum stable sets of a
%graph, SIAM J. Algebraic Discrete Methods 3 (1982) 511--522.

\bibitem{HHLP}
P. Heggernes, P. van 't Hof, D. Lokshtanov, and C. Paul, Obtaining a bipartite graph by contracting few edges, SIAM Journal on Discrete Mathematics 27 (2013) 2143--2156.

%\bibitem{IPS82}
%A. Itai, Y. Perl, Y. Shiloach, The complexity of finding maximum disjoint paths with length constraints, Networks 12 (1982) 277--286.

\bibitem{HS60}
A. J. Hoffman, R. R. Singleton, On Moore graphs with diameters 2 and 3, IBM Journal of Research and Development 4 (1960) 497--504.

\bibitem{IYN06} T. Ishii, S. Yamamoto, H. Nagamochi, Augmenting forests to meet odd diameter requirements,
Discrete Optimization 3 (2006) 154--164.

\bibitem{KBBEGRZ08}
L. Khachiyan, E. Boros, K. Borys, K. Elbassioni, V. Gurvich, G. Rudolf, J. Zhao, On short paths interdiction problems: total and node-wise limited interdiction, Theory of Computing Systems 43 (2008) 204--233.

%\bibitem{LM12}
%V. E. Levit and E. Mandrescu, Vertices belonging to all critical sets of a graph, SIAM Journal on Discrete Mathematics 26 (2012) 399--403.

\bibitem{LY80}
J.M. Lewis and M. Yannakakis, The node-deletion problem for hereditary properties is NP-complete, Journal of Computer and System Sciences, 20 (1980) 219--230.

\bibitem{LMT92}
C.-L. Li, S. T. McCormick, D. Simchi-Levi, On the minimum-cardinality-bounded-diameter and the bounded-cardinality-minimum-diameter edge addition problems, Operations Research Letters 11 (1992) 303--308.

\bibitem{LSSS20}
P. T. Lima, V. F. dos Santos, I. Sau and U. S. Souza, Reducing graph transversals via edge contractions. In Proc. of the 45th International Symposium on Mathematical Foundations of Computer Science (MFCS), volume 170 of LIPIcs, pages 64:1--64:15, 2020.

\bibitem{LM16}
P.-S. Loh, J. Ma, Diameter critical graphs, Journal of Combinatorial Theory. Series B 117 (2016) 34--58.

\bibitem{MM03}
A. R. Mahjoub, S. T. McCormick, Max flow and min cut with bounded-length paths: complexity, algorithms, and approximation,
Mathematical Programming 124 (2010) 271--284.

%\bibitem{MS13}
%M. Miller, J. Sir\'an, Moore graphs and beyond: A survey of the degree/diameter problem, The Electronic Journal of Combinatorics 20 (2013) \#DS14v2.

\bibitem{Ore68}
O. Ore, Diameter in graphs, Journal of Combinatorial Theory 5 (1968) 75--81.

\bibitem{PBP14}
F.M. Pajouh, V. Boginski and E. L. Pasiliao, Minimum vertex blocker clique problem, Networks, 64 (2014) 48--64.

\bibitem{PBP15}
F.M. Pajouh, V. Boginski and E. L. Pasiliao, Minimum edge blocker dominating set problem, European Journal of Operations Research 247 (2015) 16--26.

\bibitem{RBPDCZ10}
B. Ries, C. Bentz, C. Picouleau, D. de Werra, M.-C. Costa and R. Zenklusen, Blockers and transversals in some subclasses of bipartite graphs: when caterpillars are dancing on a grid, Discrete Mathematics 310 (2010) 132--146.

\bibitem{SBVL87}
A. A. Schoone, H. L. Bodlaender, J. Van Leeuwen, Diameter increase caused by edge deletion,
Journal of Graph Theory, 11 (1987) 409--427.

\bibitem{S68}
R. Singleton, There is no irregular Moore graph, American Mathematical Monthly 75 (1968) 42--43.

\bibitem{T10}
S. Toubaline, D\'etermination des \'el\'ements les plus vitaux pour des probl\`emes de graphes, PhD Thesis, Universit\'e Paris-Dauphine, 2010.

\bibitem{WTN83}
T. Watanabe, A.E. Tadashi and A. Nakamura,
On the NP-hardness of edge-deletion and contraction problems, Discrete Applied Mathematics 6 (1983) 63--78.

\bibitem{West}
D. B. West,
Introduction to Graph Theory, Prentice-Hall, 1996.

\bibitem{Wolk} E. Wolk, A note on ``The comparability graph of a tree,'' Proceedings of the American Mathematical Society 16 (1965) 17--20.

 \bibitem{ZRPWCB09}
R. Zenklusen, B. Ries, C. Picouleau, D. de Werra, M.-C. Costa and C. Bentz, Blockers and transversals, Discrete Mathematics 309 (2009) 4306--4314.

\end{thebibliography}
\end{document}